\documentclass[11pt]{amsart}
\usepackage[T1]{fontenc}
\usepackage[utf8x]{inputenc}
\usepackage{amsthm}
\usepackage{amssymb}
\usepackage{tikz}
\usepackage{amsmath}
\usepackage[french,english]{babel}
\usepackage{hyperref}

\usepackage[cal=boondox]{mathalfa}

\usepackage{color}

\makeatletter
\numberwithin{figure}{section}
\theoremstyle{plain}
\newtheorem{thm}{\protect\theoremname}
  \theoremstyle{plain}
  \newtheorem{lem}[thm]{\protect\lemmaname}
  \newtheorem{cor}[thm]{\protect\corollaryname}
  \theoremstyle{remark}
  \newtheorem{rem}[thm]{\protect\remarkname}
  \theoremstyle{plain}
  \newtheorem{prop}[thm]{\protect\propositionname}
  \theoremstyle{definition}
  \newtheorem{defn}[thm]{\protect\definitionname}

\makeatother

\makeatother

\usepackage{babel}
\makeatletter
\addto\extrasfrench{%
   \providecommand{\fg}{\ifdim\lastskip>\z@\unskip\fi~\frqq}%
}

\makeatother
  \addto\captionsenglish{\renewcommand{\definitionname}{Definition}}
  \addto\captionsenglish{\renewcommand{\lemmaname}{Lemma}}
  \addto\captionsenglish{\renewcommand{\propositionname}{Proposition}}
  \addto\captionsenglish{\renewcommand{\remarkname}{Remark}}
  \addto\captionsenglish{\renewcommand{\theoremname}{Theorem}}
 \addto\captionsenglish{\renewcommand{\corollaryname}{Corollary}}
  \addto\captionsfrench{\renewcommand{\definitionname}{Définition}}
  \addto\captionsfrench{\renewcommand{\lemmaname}{Lemme}}
  \addto\captionsfrench{\renewcommand{\propositionname}{Proposition}}
  \addto\captionsfrench{\renewcommand{\remarkname}{Remarque}}
  \addto\captionsfrench{\renewcommand{\theoremname}{Théorème}}
  \addto\captionsfrench{\renewcommand{\corollaryname}{Corollaire}}
  \providecommand{\definitionname}{Definition}
  \providecommand{\lemmaname}{Lemma}
  \providecommand{\propositionname}{Proposition}
  \providecommand{\remarkname}{Remark}
\providecommand{\theoremname}{Theorem}
\providecommand{\corollaryname}{Corollary}

\renewcommand{\bar}[1]{{\overline{#1}}}

\newcommand\kk{{\mathcal k}}

\newcommand\CHAR{\operatorname{char}}
\newcommand\Spec{\operatorname{Spec}}

\newcommand{\T}[1]{\rule{0pt}{#1 ex}}

\begin{document}

\selectlanguage{english}

\title[Irrationality of threefolds via Weil's conjectures]{Irrationality of generic cubic threefold \\ via Weil's conjectures}

\addtolength{\textwidth}{0mm}
\addtolength{\hoffset}{-0mm} 
\addtolength{\textheight}{0mm}
\addtolength{\voffset}{-0mm} 


\global\long\global\long\def\Alb{{\rm Alb}}
 \global\long\global\long\def\Jac{{\rm Jac}}
\global\long\global\long\def\Disc{{\rm Disc}}

\global\long\global\long\def\Tr{{\rm Tr}}
 \global\long\global\long\def\NS{{\rm NS}}
\global\long\global\long\def\PicVar{{\rm PicVar}}
\global\long\global\long\def\Pic{{\rm Pic}}
\global\long\global\long\def\Br{{\rm Br}}
 \global\long\global\long\def\Pr{{\rm Pr}}

\global\long\global\long\def\Hom{{\rm Hom}}
 \global\long\global\long\def\End{{\rm End}}
 \global\long\global\long\def\aut{{\rm Aut}}
 \global\long\global\long\def\NS{{\rm NS}}
 \global\long\global\long\def\SSm{{\rm S}}
 \global\long\global\long\def\psl{{\rm PSL}}
 \global\long\global\long\def\CC{\mathbb{C}}
 \global\long\global\long\def\BB{\mathbb{B}}
 \global\long\global\long\def\PP{\mathbb{P}}
 \global\long\global\long\def\QQ{\mathbb{Q}}
 \global\long\global\long\def\RR{\mathbb{R}}
 \global\long\global\long\def\FF{\mathbb{F}}
 \global\long\global\long\def\DD{\mathbb{D}}
 \global\long\global\long\def\LL{\mathbb{L}}
 \global\long\global\long\def\NN{\mathbb{N}}
 \global\long\global\long\def\ZZ{\mathbb{Z}}
 \global\long\global\long\def\HH{\mathbb{H}}
 \global\long\global\long\def\Gal{{\rm Gal}}
 \global\long\global\long\def\OO{\mathcal{O}}
 \global\long\global\long\def\pP{\mathfrak{p}}
 \global\long\global\long\def\pPP{\mathfrak{P}}
 \global\long\global\long\def\qQ{\mathfrak{q}}

\author{Dimitri Markushevich, Xavier Roulleau}
\begin{abstract}
\sloppy
An arithmetic method of proving the irrationality of smooth projective 3-folds is described, using reduction modulo $p$. It is illustrated by an application to a cubic threefold, for which the hypothesis that its intermediate Jacobian is isomorphic to the Jacobian of a curve is contradicted by reducing modulo 3 and counting points over appropriate extensions of $\FF_3$.
As a spin-off, it is shown that the 5-dimensional Prym varieties arising as intermediate Jacobians of certain cubic 3-folds have the maximal number of points over $\FF_q$ which attains \mbox{Perret's} and Weil's upper bounds.
\end{abstract}

\maketitle

\section{Introduction}

In early 70's, Clemens and Griffiths proved that a
smooth complex cubic threefold is non-rational, that is, not birational to $\PP^{3}$. The cubic threefold was thus one of the
first counterexamples to the famous Lüroth problem, obtained almost at the same time as the counterexamples of Iskovskikh--Manin and Artin--Mumford. The main tool in the approach of Clemens--Griffiths is the following criterion:
\begin{thm}[Clemens--Griffiths \protect{\cite[Corollary 3.26]{Clemens}}]
\label{thm:(Clemens-Griffiths,-).}
Let $V_{/\CC}$ be a smooth threefold with $h^{3,0}(V)=0$. Let $J(V)$
be the intermediate Jacobian of $V$. If $V$ is birational to $\PP^{3}$
then $J(V)$ is isomorphic (as a principally polarized abelian variety)
to the Jacobian of a (possibly reducible) non-singular curve.
\end{thm}
Then, the core and the difficult part of \cite{Clemens} is 
the following result:
\begin{thm}[\protect{\cite[Theorem 13.12]{Clemens}}]
\label{thm:-J(X) not jacob} The
intermediate Jacobian of a smooth cubic threefold $X$ is not isomorphic
to the Jacobian of a curve.
\end{thm}
Since then several alternative proofs of nonrationality of cubic threefolds have appeared. 
Mumford (in the Appendix to \cite{Clemens}) and Tyurin \cite{Tyu1972} reproved Theorem 2 using Prym varieties.
Murre \cite{Murre} extended their approach to smooth cubic threefolds over any field of characteristic $\neq 2$, working
with the Prym variety which is an algebraic representative of the Chow group $A^2(X)$ in place of the intermediate Jacobian.

Other authors produced new arguments proving the result only for generic cubic threefolds. Though these approaches give a weaker result, they are still of great interest in view of complexity of the existing proofs of the full original result, and also in view of possible applications to other types of varieties, for which the answer to the question of rationality remains unknown. Beauville \cite{Beau} showed that the structure of the automorphism group of some special smooth cubic threefolds is incompatible with the hypothesis that their intermediate Jacobians are Jacobians of curves. The authors of \cite{Bard},  \cite{Coll}, \cite{Gwena} used the method of degeneration to a singular cubic threefold and showed that the limit generalized intermediate Jacobian (in the sense of Zucker) is not isomorphic to the generalized Jacobian of a stable curve. As the Jacobian locus in the compactified moduli space of principally polarized abelian varieties (p.p.a.v.) is closed, these results for special cubic threefolds imply the statement of Theorem 2 and hence the non-rationality for a generic cubic threefold.

In the present paper, we give yet another approach to the proof of the statement of Theorem \ref{thm:-J(X) not jacob}
for a special threefold. This approach is arithmetic and uses the reduction modulo a prime.
Namely, we prove:
\begin{thm}
\label{thm:A-generic-complex}
There exists a cubic threefold $X_{/\ZZ}$  with good reduction modulo $3$, such that the reduction $J(X)_{/\FF_{3}}$ of the intermediate Jacobian of $X$
is absolutely simple and is not isomorphic to the Jacobian of a curve over any finite extension of $\FF_3$. For such $X$, also the intermediate Jacobian of $X_{/\CC}$ is absolutely simple and is not isomorphic to the Jacobian of a curve.
\end{thm}
After some preliminary material presented in Section 2, we provide in Section 3 an explicit example of a cubic threefold for which $J(X)_{/\FF_{3}}$  is not isomorphic to the Jacobian of a curve.
The proof is done by verifying that if such a curve existed, it
would have too few points on the field $\mathbb{F}_{3^r}$.  By Murre
 \cite{Murre}, this result implies the irrationality of the particular
 cubic $X_{/\FF_{3}}$ over~$\overline{\FF}_3$.

This example is also interesting in the context of the following two natural questions, asked for a
given field $\kk$ and a positive integer $n$: 1)~does there exist a p.p.a.v. $A_{/\kk}$ of dimension $n$ which is not
isogenous over $\bar\kk$ to the Jacobian $J(C)$ of any curve $C_{/\kk}$ defined over~$\kk$? 2)~Does there exist an absolutely simple abelian variety of dimension $n$ defined over $\kk$? The answers are affirmative when $\mathcal k$ is uncountable and algebraically closed by trivial reasons. Chai and Oort \cite{Chai} answered the first question in affirmative for any $n\geq 4$ for $\kk=\bar\QQ$ (see also Tsimerman \cite{Tsimerman}) and remarked that the question remains open over the countable fields $\bar\FF_p$. Howe and Zhu answered the second question in affirmative for any field and any $n$ in \cite{HZ}.
The intermediate Jacobian of our particular cubic is an explicit example providing affirmative answers to both questions for $n=5$ and $\kk=\FF_3$ or $\FF_9$ (Proposition \ref{thm:J(X) not jacob fiinite field} and Corollary \ref{Chai Oort}).

In Section 4 we explain how the result of Theorem 3 over $\overline{\FF}_3$ implies the one over $\bar \QQ$ and $\CC$. The proof passes through semistable reduction and Néron models.

In Section 5, we give one more application of our approach:
we obtain new results on the maximal number of points on $5$-dimensional
Prym varieties that give a partial answer to a question raised in
\cite{Aubry}.

\textbf{Acknowledgements.} The authors thank the referee for useful remarks. The first author thanks Vasily Golyshev
for sharing his ideas about relations between rationality questions
for Fano varieties and number theory. The second author thanks Michele Bolognesi for discussions and his interest in the
paper, and also thanks Bernd Ulrich and Dino Lorenzini for communicating a
proof of Proposition~\ref{prop:not iso jacob}.

\section{Preliminaries.}

In this section, we introduce notation and recall some known
results and tools that will be used later.

\subsection{Zeta functions and Weil polynomials}

Let $\kk=\FF_{q}$ be the finite field with $q$ elements, and $\bar{\kk}$
its algebraic closure. For a variety $X$ defined over $\FF_{q}$,
let $\bar{X}=X\otimes_{\kk}\bar{\kk}$. Let $r\in\NN^{*}$; we denote by $N_{r}(X)$ the number of $\FF_{q^{r}}$-points on
$X$. The zeta function of $X$ is defined by 
\[
Z(X,T)=\exp\left(\sum_{r\geq1}\frac{N_{r}(X)}{r}T^{r}\right).
\]
The $j$-th Weil polynomial
\[
Q_{j}(X,T)=\det\big(T-F^*|H^{j}(\bar{X},\QQ_{\ell})\big)
\]
is the characteristic polynomial of the Frobenius $F$ acting on the
$j$-th étale cohomology group $H^{j}(\bar{X},\QQ_{\ell})$. The Weil
Conjectures proved by Dwork, Grothendieck and Deligne (see historical comments in \cite[Appendix C]{Hartshorne})
tell us that, if $X$ is smooth and projective, then the Weil polynomial $Q_{j}$ has integer coefficients,
does not depend on the prime $\ell$, provided $\ell\neq p=\CHAR(\FF_{q})$, and the zeta function of $X$ satisfies the equality
\[
Z(X,T)=\prod_{i=0}^{2\dim X}P_{j}(X,T)^{(-1)^{j+1}},
\]
where $P_{i}(X,T)=T^{\deg Q_{i}}Q_{i}(X,\frac{1}{T})$. Moreover,
the roots of $Q_{i}$ are algebraic integers of absolute value $q^{i/2}$. 

\subsection{Cubic threefolds}

Let $X$ be a smooth cubic $3$-fold over $\FF_{q}$ containing a
line defined over $\FF_{q}$. For $r\in\NN^{*}$, let us define $M_{r}(X)$
by 
\begin{equation}\label{MrX}
M_{r}(X):=\frac{1}{q^{r}}\big(N_{r}(X)-(1+q^{r}+q^{2r}+q^{3r})\big).
\end{equation}
One has $M_{r}(X)=-\sum_{j=1}^{10}\omega_{j}^{r},$ where the numbers
$q\omega_{j},\,j=1,\ldots,10$ are the roots of $Q_{3}(X,T)$ (see e.g.
\cite[Section 4]{DLR}). Let $F(X)$ be the Fano surface of lines
on $X$. Then we have (see e.g. \cite[Theorem 4.1]{DLR}): 
\[
Q_{1}(F(X),T)=\prod_{i=1}^{10}(T-\omega_{i}).
\]
Therefore $-M_{1}(X)$ equals the trace of the Frobenius action on
$H^{1}(\overline{F(X)},\QQ_{\ell})$. The Albanese variety
$\Alb(F(X))$ of $F(X)$ is $5$-dimensional. Under our assumption that $F(X)$ contains
a $\FF_{q}$-rational point, say $z_0$, we have the Abel--Jacobi map $\alpha_{z_0}:F(X)\to \Alb(F(X))$ defined over $\kk=\FF_{q}$. According to Beauville \cite{bea2}, $\alpha_{z_0}$ is an embedding and $S-S$ is a theta-divisor of
a principal polarization of $\Alb(F(X))$, where $S=\alpha_{z_0}(F(X))$, so that $\Alb(F(X))$ possesses a principal polarization $\Theta$ defined over $\kk$.

Though classically the intermediate Jacobian of a smooth projective threefold $X$
is defined via the Hodge theory over $\CC$, in the case when $X$ is a smooth cubic 3-fold, one can obtain its intermediate Jacobian by a purely algebraic construction valid over any field. As Murre proves in \cite{MurreSLN412}, one can choose the Albanese variety $(\Alb(F(X)),\Theta)$ as such a construction. In the sense of \cite{Murre1985},  $\Alb(F(X))$ is an algebraic representative of the Chow group $A^2(X)$ in the same way as $J(X)$ over $\CC$. We will denote $\Alb(F(X))$ by $J(X)$ and call it the intermediate Jacobian of $X$ whatever the base field is.  If $X$ is the reduction mod $p$ of a cubic threefold $X'_{/\ZZ}$, then $F(X')$, $J(X')$ are viewed as schemes over $\ZZ$ and $F(X)$, respectively $J(X)$ are their reductions mod~$p$.

\subsection{\label{subsec:Another-method-for} Some results on curves, Jacobians and abelian varieties}
We recall the formulas for the numbers of points
on an abelian variety and a curve which are consequences of the Lefschetz trace formula \cite{Milne}
for the Frobenius endomorphism:
\begin{lem}
\label{lemma:NumbPointsAbelian}For an abelian variety $A_{/\FF_{q}}$,
one has 
\[
N_{1}(A_{/\FF_{q}})=Q_{1}(A_{/\FF_{q}},1).
\]
\end{lem}

\begin{lem}
\label{lem:number points curve}The number of $\FF_{q}$-rational
points on a smooth curve $C_{/\FF_{q}}$ is 
\[
N_{1}(C)=q+1-\tau,
\]
where $\tau$ is the trace of the Frobenius endomorphism acting on $H^{1}(\overline{C},\QQ_{\ell})$. 
\end{lem}
\begin{rem}\label{N1 M1}
Since $H^{1}(\overline{C},\QQ_{\ell})$ is isomorphic
to $H^{1}(\overline{J(C)},\QQ_{\ell})$  as a Galois module, the number $\tau$ is also
the trace of Frobenius on $H^{1}(\overline{J(C)},\QQ_{\ell})$.
Therefore if the intermediate Jacobian of a cubic threefold $X_{/\FF_{q}}$
is isomorphic to the Jacobian of a curve $C$, one has $$N_{1}(C)=q+1+M_{1}(X),$$
where $M_{1}(X)$ is the degree $9$ coefficient of $Q_{1}(F(X),T)=Q_{1}(J(X),T)$. 
\end{rem}

Recall that an abelian variety $A$ over a field $\kk$ is said to be  absolutely
simple if $\bar{A}=A\times_{\kk}\bar{\kk}$ is simple, that is, has no proper abelian subvarieties. As a proper abelian subvariety is always a factor of a decomposition into a direct sum modulo isogeny,
a non-simple abelian variety has a pair of orthogonal idempotents in
its endomorphism ring tensored with $\QQ$. Thus an absolutely simple abelian variety can be characterized by the property that the ring $\End(\bar A)$ has no zero divisors. We will use the following criterion for absolute simplicity, which is
a consequence of \cite[Proposition 3 (1)]{HZ} and was also
 used in \cite[Proposition 4.17]{DLR}:
\begin{prop}
\label{prop:Criteria Abs simple}  Let $A$ be a $d$-dimensional
abelian variety over $\FF_{q}$.
Suppose that the polynomials $Q_{1}(A_{/\FF_{q^r}},T)$ are irreducible
and that they are not elements of the ring $\ZZ[T^{k}]$ for any $k\geq2$ and any
integer $r$ such that $\varphi(r)|2d$, where $\varphi$ denotes the Euler totient function. Then $A$ is absolutely simple.
\end{prop}

\begin{proof}
By \cite[Proposition 3]{HZ}, in order to prove that $A$ is absolutely simple it is sufficient to prove that:\\
(a) there is no $d>1$ such that the characteristic polynomial of the Frobenius is in $\mathbb{Z}[T^d]$, and\\
(b) there is no $d>1$ and no primitive $d$-th root of unity $\zeta$ such that $\mathbb{Q}(\pi^d)$ is a proper sub-field of $\mathbb{Q}(\pi)$ and $\mathbb{Q}(\pi)=\mathbb{Q}(\pi^d,\zeta)$.\\
If (b) does not hold and $\mathbb{Q}(\pi)=\mathbb{Q}(\pi^d,\zeta)$, then the degree $\varphi(d)$ of the extension $\mathbb{Q}(\zeta)/\mathbb{Q}$ must divide $\deg \mathbb{Q}(\pi)=2 \dim A$. Thus, to verify the absolute simplicity, it is sufficient to check that $\mathbb{Q}(\pi^d)=\mathbb{Q}(\pi)$ for all $d$ such that $\varphi(d)$ divides $2 \dim A$. The latter equality is true if the degree of the minimal polynomial of $\pi^d$ is $2\dim A$, i.e. if $Q_1(A_{\mathbb{F}_{q^r}},T)$ is irreducible.
\end{proof}


\section{Example of a cubic over $\ZZ$}
\label{first example}


Let us consider the smooth complex cubic threefold defined
by the following equation with integer coefficients:

\setlength{\arraycolsep}{0pt}
\begin{equation}
\begin{array}{lr}
\T3 X_{/\ZZ}= & \big\{x_{1}^{3}+2x_{1}^{2}x_{2}+2x_{1}x_{2}^{2}+x_{1}^{2}x_{3}+2x_{1}x_{2}x_{3}+2x_{1}x_{3}^{2}+2x_{2}x_{3}^{2}+x_{3}^{3}+x_{1}^{2}x_{4}\\
 & \ +2x_{1}x_{2}x_{4}+x_{2}^{2}x_{4}+x_{2}x_{3}x_{4}+x_{1}x_{4}^{2}+2x_{3}x_{4}^{2}+x_{4}^{3}+x_{2}^{2}x_{5}+2x_{2}x_{3}x_{5}\\
 & \ +2x_{3}^{2}x_{5}+x_{1}x_{4}x_{5}+x_{2}x_{4}x_{5}+x_{4}^{2}x_{5}+x_{2}x_{5}^{2}+2x_{4}x_{5}^{2}+x_{5}^{3}=0\big\}.
\end{array}\label{eq:X}
\end{equation}
We denote by $X_{/\FF_{3}}$ its reduction mod $3$.
Using the computer algebra system \cite{Macaulay2}, one can easily verify that $X$ is smooth over $\bar\FF_{3}$, that $\#X(\FF_3)=22$, and that exactly one quadruple of the $\FF_3$-points of $X$ is aligned, so that
$X_{/\FF_{3}}$ contains one line defined over $\FF_3$.
We denote by $J(X)_{/\FF_{3}}$ the reduction of $J(X)$, naturally isomorphic to $J(X_{/\FF_{3}})$. 
\begin{thm}
\label{thm:J(X) not jacob fiinite field}The abelian variety $J(X)_{/\FF_{3}}$
is absolutely simple and is not isomorphic to the Jacobian of a curve
over any finite extension of $\FF_{3}$.
\end{thm}

We start by checking the absolute simplicity.
A computation by
the algorithm described in \cite[Section 4.3]{DLR} yields:

\begin{multline*}
Q_{1}(J(X_{/\FF_{3}}),T)=  243-486T+405T^{2}-90T^{3}-123T^{4}+125T^{5}\\
  -41T^{6}-10T^{7}+15T^{8}-6T^{9}+T^{10}.
\end{multline*}

Consistently with earlier computations, the Weil conjectures tell us that $X_{/\FF_{3}}$ contains $22$ points
and $1$ line defined over $\FF_{3}$. To obtain the Weil polynomial $Q_{1}(J(X)_{/\FF_{3^{r}}},T)$ over $\FF_{3^{r}}$, one can use the formula 
\[
Q_{1}(J(X)_{/\FF_{3^{r}}},T)=\prod_{j=1}^{10}(T-\omega_{j}^{r}),
\]
where the $\omega_{j}$'s are the roots of $Q_{1}(J(X)_{/\FF_{3}},T)$.
A computation shows that the Weil polynomials 
\[
Q_{1}(J(X)_{/\FF_{3^r}},T),\,r\in\{2,3,4,6,11,22\}
\]
are irreducible and are not in the ring $\ZZ[T^{n}]$ for any $n\geq2$.
So, by Proposition \ref{prop:Criteria Abs simple}, the abelian variety
$J(X)_{/\FF_{3}}$ is absolutely simple.

Now we will prove that $J(X)_{/\FF_{3}}$ cannot be the Jacobian of
a curve over any finite extension $\FF_{3^{r}}$ of $\FF_{3}$. 
We will use the following result:
\begin{thm}[Serre \protect{\cite[Théorème 9]{LauterSerre}}]
\label{thm:(Serre-)-Suppose}
Let $(A,\Theta)$ be a p.p.a.v over a finite field $\FF_{q}$. Let
$C$ be a curve defined over $\FF_{q^{r}}$ and let $J(C)$ be its
Jacobian, endowed with its principal polarization.
Suppose that the p.p.a.v. $(A,\Theta)_{/\FF_{q^r}}$
and $J(C)$ are isomorphic over $\FF_{q^{r}}$. Then there exists
a curve $C'$ defined over $\FF_{q}$ such that $C'_{/\FF_{q^{r}}}=C'\otimes_{\FF_{q}}\FF_{q^{r}}$
is isomorphic to $C$, and either:\\
a) $J(C')\simeq (A,\Theta)$ over $\FF_{q}$ or,\\
b) $J(C')$ is a quadratic twist of $(A,\Theta)$ and $J(C')_{/\FF_{q^{2}}}\simeq (A,\Theta)_{/\FF_{q^{2}}}$.
\end{thm}
Therefore, in order to prove that the intermediate Jacobian $J(X)_{/\FF_{3}}$
is not isomorphic to the Jacobian of a curve over any finite extension of $\FF_{3}$,
it suffices to prove this for curves defined over $\FF_{3}$ and $\FF_{9}$.
So, let us assume that $J(X)_{/\FF_{3}}$ is isomorphic over $\FF_{3^r}$
to a product of Jacobians of curves. Then, since we know that $J(X)_{/\FF_{3}}$
is absolutely simple, $J(X)_{/\FF_{3^r}}$ is isomorphic to $J(C)_{/\FF_{3^r}}$ for just one smooth irreducible curve $C_{/\FF_{3}}$ of genus $5$, and we may assume that $r=1$ or 2. 

By Remark \ref{N1 M1}, we have
\[
N_{r}(C)=1+q^{r}+M_{r}(X),
\]
and therefore $M_{r}(X)\geq-1-q^{r}$ for all $r\geq1$. Furthermore, by \eqref{MrX} with $q=3$, we have
\[
M_{1}(X)=(22-(1+3+9+27))/3=-6< -1-3,
\]
hence $r> 1$, that is we are not in the case a) of Theorem \ref{thm:(Serre-)-Suppose}.

Thus it remains to consider the case b): we are assuming now that there exists a curve $C'_{/\FF_{3}}$
such that $J(X)$ is a quadratic twist of $J(C')$. Then by \cite[Théorème 9]{LauterSerre},
$Q_{1}(C',T)=Q_{1}(J(X),-T)$ and the curve $C'$ has 
\[
N_{1}(C')=1+3+6=10
\]
points over $\FF_{3}$. Let $\omega_{1},\dots,\omega_{10}$ be the
roots of $Q_{1}(J(X)_{/\FF_{3}},T)$. Using the relation
\[
Q_{1}(J(X)_{/\FF_{9}},T)=\prod_{i=1}^{10}(T-\omega_{i}^{2}).
\]
we compute 
\begin{multline*}
Q_{1}(J(X)_{/\FF_{9}},T)=  T^{10}-6T^{9}+23T^{8}-76T^{7}+221T^{6}-535T^{5}\\
  +1989T^{4}-6156T^{3}+16767T^{2}-39366T+59049.
\end{multline*}
By our assumption, $Q_{1}(C_{/\FF_{9}},T)=Q_{1}(J(X)_{/\FF_{9}},T)$.
Therefore, since $M_{2}(X)=-6$, we get 
\[
N_{2}(C')=1+3^{2}-6=4.
\]
This is absurd since the number of points over $\FF_{9}$
should be larger than the number of points over $\FF_{3}$: $N_{2}(C')\geq N_{1}(C')$.
Therefore $J(X)$ is not isomorphic to the Jacobian of a curve over any finite
extension of $\FF_{3}$, and this finishes the proof of Theorem \ref{thm:J(X) not jacob fiinite field}.

\begin{cor}
\label{thm:Irrat over finite field}The cubic $X_{/\FF_{3}}$ is irrational.
\end{cor}
\begin{proof}
This follows from \cite[Theorem 3.1.1]{Murre}, which is the analog
of Theorem \ref{thm:(Clemens-Griffiths,-).} for cubic threefolds valid over finite fields.

\end{proof}

Our argument indeed proves a slightly stronger result:

\begin{cor}\label{Chai Oort}
The intermediate Jacobian $J(X)_{/\FF_{3}}$ of $X_{/\FF_{3}}$
is not isogenous over $\bar\FF_3$ to the Jacobian of any curve defined over $\FF_{3}$ or $\FF_{9}$.
\end{cor}

\begin{proof}
By Honda--Tate Theorem \cite{Tate}, isogenous abelian varieties have the same Weil polynomials, and we proved that the first Weil polynomial of $J(X)_{/\FF_{3}}$ cannot be the Weil polynomial of the Jacobian of a curve defined over  $\FF_{3}$ or $\FF_{9}$.

\end{proof}

\section{From \protect$J(X)\protect_{\protect\FF\protect_3}$ to \protect$J(X)\protect_{\protect\CC}$}

Let $X$ be the smooth complex cubic threefold defined by equation (\ref{eq:X}). We have seen that 
$J(X)_{\FF_3}$ is absolutely simple and is not isomorphic to the Jacobian of a curve over $\bar\FF_3$. We will show how one deduces from this the same properties for $J(X)_{\CC}$ over $\CC$. 

The following result holds:

\begin{prop}
\label{abs-simple}
Let $U\subset\Spec\ZZ$ be an open subscheme and $A_{/U}$ an abelian scheme of relative dimension $g>0$. Assume that for some prime $p\in \ZZ$ with $(p)\in U$, the fiber $A_{(p)}
=A_{/\FF_p}$ is absolutely simple. Then the generic fiber $A_0=A_{/\QQ}$ and the complex abelian variety $A_{/\CC}=A_0\otimes_{\QQ}\CC$ are also absolutely simple.
\end{prop}

\begin{proof} 
The statement for $A_{/\QQ}$ follows from \cite[Lemma 6]{ChaiO}. If we now assume that $A_{/\CC}$ is non-simple, then $A_{/\CC}$ is isogenous to the product of two abelian varieties of smaller dimension. By Weil descent, the isogeny specializes to that defined over a number field, which is impossible by what we have proved for $A_{/\QQ}$.
\end{proof}

\begin{prop}
\label{prop:not iso jacob}
As in Proposition \ref{abs-simple}, let $U\subset\Spec\ZZ$ be an open subscheme and 
$A_{/U}$ an abelian scheme of relative dimension $g>0$ carrying a principal polarization. Assume that for some prime $p\in \ZZ$ with $(p)\in U$, the fiber $A_{(p)}
=A_{/\FF_p}$ is an absolutely simple abelian variety which is not isomorphic to the Jacobian of a smooth curve as a p.p.a.v.
over $\bar\FF_p$. Then the generic fiber $A_0=A_{/\QQ}$ (resp. $A_{/\CC}=A_0\otimes_{\QQ}\CC$) is also absolutely simple and is not isomorphic to the Jacobian of a smooth curve over $\bar\QQ$ (resp. over~$\CC$).
\end{prop}

\begin{proof} 
The assertion over $\CC$ is reduced to that over $\bar\QQ$ by the standard argument using Weil descent, so let us prove the assertion over $\bar\QQ$.
Thus, we are assuming that $A_{\bar\QQ}$ is isomorphic to a product of Jacobians of smooth
curves. By Proposition \ref{abs-simple}, $A_{\bar\QQ}$ is absolutely simple, hence there is just
one smooth irreducible curve $C$, defined over some number field $K$, such that
$A_{/K}\simeq J(C_{/K})$.

By \cite[Theorem 2.4]{DeMu69}, the existence of a semistable reduction for $ J(C_{/K})$ implies that
$C_{/K}$ has a semistable reduction at $p$, therefore $C$ has a regular model whose fiber at $p$ is
a semi-stable curve $C_{p}$ defined over some finite extension of $\FF_p$. The connectedness
of the special fiber of the semistable reduction for $ J(C_{/K})$ implies that the relative Jacobian $J(C)$
is a Néron model. Moreover, $J(C_{/K})$ even has a good, or abelian reduction at $p$. According to Example 9.2.8 and Theorem 9.4.4 from \cite{BLR}, the curve $C_{p}$ has smooth components with normal crossings such that:\\
a) the graph of the intersection matrix of the components is a tree;\\
b) the sum of the genera of the irreducible components equals the
genus of the generic fiber;\\
c) the special fiber of the Jacobian is isomorphic to the product
of the Jacobians of the components.\\
Since the special fiber of $J(C)$, isomorphic to $A_{/\FF_p}$,  is actually absolutely simple, the special fiber of $C$ is a smooth genus $g$ component with possibly some trees of smooth rational components attached to it. The trees of rational components do not influence the Jacobian, and we conclude that $J(X)_{/\FF_{p}}$ is the
Jacobian of a curve, which contradicts our hypotheses. This ends the proof. 
\end{proof}

From Theorem \ref{thm:J(X) not jacob fiinite field} and Propositions \ref{abs-simple}, \ref{prop:not iso jacob} we deduce:

\begin{cor}
\label{thm:Irrat over C} 
For the cubic threefold $X$ defined by equation (\ref{eq:X}), the intermediate Jacobian $J(X)_{\CC}$ is not
isomorphic to a product of Jacobians of curves.
\end{cor}

Since the Jacobians and their products form a closed subvariety of
the moduli space of p.p.a.v.,  this implies that $J(X)_{\CC}$ is not
isomorphic to a product of Jacobians of curves for a generic cubic threefold $X$ and,
by Theorem 1, the generic cubic threefold is irrational. We have thus finished the proof of Theorem 3.

\section{Examples of five-dimensional Prym varieties with maximal number of
points}

For an abelian variety $A_{/\FF_{q}}$ we denote by $M_{1}(A)$, as in Section 3, the trace of the Frobenius endomorphism of $H^{1}(\overline{A},\QQ_{\ell})$ taken with opposite sign. 

Let $C$ be a smooth curve of genus $g$ and let $C'\to C$ be an
étale double cover. Let $\Pr(C'/C)$ be the $(g-1)$-dimensional
Prym variety associated to $C'\to C$. For Prym varieties, $M_{1}\big(\Pr(C'/C)\big)$
has the following interpretation: 
\[
M_{1}\big(Pr(C'/C)\big)=N_{1}(C')-N_{1}(C).
\]
In \cite[Theorem 2, (ii)]{Perret}, Perret gives upper bounds on the
number of points on a Prym variety:
\begin{thm}
One has:
\begin{equation}
N_{1}\big(\Pr(C'/C)\big)\leq\left(q+1+\frac{M_{1}(\Pr(C'/C))}{g-1}\right)^{g-1}.\label{eq:Perret Bound}
\end{equation}
\end{thm}
Let $A$ be a $(g-1)$-dimensional abelian variety. The Weil's bound
gives us: 
\begin{equation}
\left(q+1+\frac{M_{1}(A)}{g-1}\right)^{g-1}\leq\big(q+1+[2\sqrt{q}]\big)^{g-1}.\label{eq:Weilbound}
\end{equation}
In \cite{Aubry} Y. Aubry and S. Haloui defined the following quantities
\begin{defn}
For any integer $g>1$, set 
\[
\Pr{}_{q}(g)=\max_{(C',C)}N_{1}\big(\Pr(C'/C)\big)
\]
where the maximum is taken over all smooth curves $C$ of genus $g$
over $\FF_{q}$ and their double étale covers $C'$.
\end{defn}
In \cite[Corollary 17]{Aubry}, they computed $\Pr{}_{q}(3)$. We
are giving here the value of $\Pr{}_{q}(6)$ for some powers of primes $q$.

Let $X_{/\FF_{q}}$ be a smooth cubic threefold. The intermediate
Jacobian $J(X)$ of $X$ is a Prym variety associated to a plane quintic
curve (thus a curve of genus $6$), see e.g. \cite{Clemens}. Consider the Fermat cubic threefold $X=\{x_{1}^{3}+\dots+x_{5}^{3}=0\}$.
We have:
\begin{prop}[\protect{\cite[Proposition 4.12]{DLR}}]
 If $p\equiv2\pmod3$, then 
\[
Q_{1}(J(X)_{\FF_{p}},T)=(T^2+p)^{5}.
\]
If $p\equiv1\pmod3$, one can write uniquely $4p=a^{2}+27b^{2}$ with
$a\equiv1\pmod3$ and $b>0$, and 
\[
Q_{1}(J(X)_{\FF_{p}},T)=(T^2+aT+p)^{5}.
\]
\end{prop}
Suppose that $p\equiv2\pmod3$ and let $q$ be an even power of $p$.
Then  
\[
Q_{1}(J(X)_{\FF_{q}},T)=(T^2+2\sqrt{q}T+q)^{5}
\]
and 
\[
N_{1}(J(X))=(q+1+2\sqrt{q})^{5},
\]
thus the equality in both inequalities \ref{eq:Perret Bound} and
\ref{eq:Weilbound} is attained. The example of the Fermat cubic threefold
implies that:
\begin{thm}
Let $p$ be a prime such that $p\equiv2\pmod3$. For $q$ an even
power of $p$, one has: 
\[
\Pr_{q}(6)=\big(q+1+[2\sqrt{q}]\big)^{5}.
\]
\end{thm}
The Klein cubic threefold $X_{Kl}\subset\PP_{\ZZ}^{4}$ is defined
by the equation 
\begin{eqnarray}
x_{1}^{2}x_{2}+x_{2}^{2}x_{3}+x_{3}^{2}x_{4}+x_{4}^{2}x_{5}+x_{5}^{2}x_{1}=0\label{equ3}
\end{eqnarray}
 It has good reduction at every prime $p\ne11$. By \cite[proposition 4.15]{DLR},
if $-11$ is not a square modulo the prime $p$, then $J(X_{Kl/\FF_{p}})$
is isogenous to $E^{5}$ where $E$ is a supersingular elliptic curve.
Thus, using Gauss quadratic reciprocity law, we obtain:
\begin{thm}
Let $p$ be a prime such that $p\equiv1,3,4,5,9\,\pmod{11}$. For
any even power $q$ of $p$ one has 
\[
\Pr_{q}(6)=(q+1+[2\sqrt{q}])^{5}.
\]
\end{thm}
\begin{rem}
The method in Sections 3 and 4 cannot be applied to the Klein or Fermat cubic
threefolds, since their intermediate Jacobians are isogenous
to a product of elliptic curves. However, using Hurwitz bound on the
number of automorphisms of a curve, Beauville proves in \cite[3.3, Theorem 3]{Beau}
that the intermediate Jacobian of $X_{Kl}$ cannot be isomorphic to the
Jacobian of a curve as a p.p.a.v. 
\end{rem}


\vspace{5mm}
 
\noindent Dimitri Markushevich,\\ 
Universit\'e Lille-1, \\
Laboratoire Paul Painlevé\\
59655 Villeneuve d'Ascq Cedex,\\
France\\
{\tt markushe@math.univ-lille1.fr}
\vspace{5mm}

\noindent Xavier Roulleau,\\
Universit\'e d'Aix-Marseille,\\
CNRS, Centrale Marseille, \\
I2M UMR 7373,\\
13453 Marseille\\
France\\
{\tt xavier.roulleau@univ-amu.fr}\\

\selectlanguage{english}%

\end{document}